\documentclass[11pt]{amsart}

\usepackage{hyperref}
\usepackage{amsfonts,mathrsfs,bbm,latexsym,rawfonts,amsmath,amssymb,amsthm,epsfig}
\usepackage{fullpage, setspace, color}
\newtheorem{thm}{Theorem}[section]

\newtheorem{rem}[thm]{Remark}
\newtheorem{lem}[thm]{Lemma}

\numberwithin{equation}{section}

 \newcommand{\al}{\alpha}
 \newcommand{\be}{\beta}
 \newcommand{\ld}{\lambda}
 
 \newcommand{\de}{\delta}

 \newcommand{\om}{\omega}
 
 \newcommand{\ga}{\gamma}

 \newcommand{\sch}{Schr\"odinger }

 \renewcommand{\P}{\mathcal{P}}
 
 \renewcommand{\S}{\mathscr{S}}

 \DeclareMathOperator{\tr}{tr}

 \newcommand{\Real}{\mathbb{R}}

 \newcommand{\norm}[1]{\Vert#1\Vert}
 
 \def\<{\left\langle} \def\>{\right\rangle}
 \def\({\left(} \def\){\right)}
 \newcommand{\n}{\nabla}
 \newcommand{\p}{\partial}

 \renewcommand{\t}[1]{\tilde{#1}}
 \renewcommand{\b}[1]{\bar{#1}}

 \newcommand{\ii}{\mathbf{i}}
 \newcommand{\jj}{\mathbf{j}}
 \newcommand{\kk}{\mathbf{k}}

   \newcommand{\tn}{\tilde{\nabla}}
    \newcommand{\bn}{\bar{\n}}

\allowdisplaybreaks

\begin{document}

\title{Uniqueness of Schr\"odinger Flow on Manifolds}

\author{Chong Song, Youde Wang}

\address{School of Mathematical Sciences, Xiamen University, Xiamen, 361005, P.R.China.}
\email{songchong@xmu.edu.cn}

\address{Academy of Mathematics and Systematic Sciences, Chinese Academy of Sciences, Beijing 100190, P.R. China}
\email{wyd@math.ac.cn}

\thanks {The first author is supported by Natural Science Foundation of
Fujian Province of China (No. 2014J01023) and Fundamental Research Funds for the Central University (No. 20720170009). The second author is supported by NSFC (No. 11471316).}

%\date{04-21-2017}

\begin{abstract}
In this paper, we show the uniqueness of \sch flow from a general complete Riemannian manifold to a complete K\"ahler manifold with bounded geometry. While following the ideas of McGahagan\cite{Mc}, we present a more intrinsic proof by using the distance functions and gauge language.
\end{abstract}

\maketitle

\section{Introduction}

The \sch flow, which is independently introduced in \cite{DW1} and \cite{TU}, is a geometric Hamiltonian flow of maps between manifolds. Suppose $M$ is a Riemannian manifold, $N$ is a K\"ahler manifold with complex structure $J$ and $u_0$ is a map from $M$ to $N$. The \sch flow is a time-dependent map $u:[0,T)\times M\to N$ satisfying the equation
\begin{equation}\label{e:sch0}
\left\{\begin{aligned}
\p_t u&=J(u)\tau(u),\\
u(0)&=u_0,
\end{aligned}\right.
\end{equation}
where $\tau(u)$ is the tension field of $u$.

The \sch flow is a natural generalization of the Laudau-Lifshitz equation which emerges from the study of ferromagnetism~\cite{LL}. It is also closely related to the Da Rios equation which models the locally induced motion of a vortex filament~\cite{Da}. The PDE aspects of the \sch flow, including local well-posedness, global regularity and blow-up phenomena, have been intensively studied in the last two decades. We refer to \cite{Ding,BIKT,MRR,RRS} and references therein for various results.

The local existence of the \sch flow from a general Riemannian manifold into a K\"ahler manifold was first obtained by Ding and Wang~\cite{DW2}. By using a parabolic approximation and the geometric energy method, they proved that, if $M$ is an $m$ dimensional compact Riemannian manifold or the Euclidean space $\Real^m$ and the initial map $u_0\in W^{k,2}(M,N)$ with $k\ge[m/2]+2$, then there exists a local solution $u\in L^\infty([0,T),W^{k,2}(M,N))$. When the domain manifold $M=\Real^m$, same existence result was reproved by McGahagan~\cite{Mc} by using a wave map approximating scheme.

The uniqueness of the \sch flow turn out to be a more delicate issue. In~\cite{DW1}, Ding and Wang proved the uniqueness of $C^3$-solutions to the \sch flow when $M$ is compact. It follows from their proof that, when $M$ is compact or the Euclidean space $\Real^m$,  a local solution to the \sch flow in the space $L^\infty([0,T], W^{[m/2]+4,2}(M,N))$ is unique. Their approach is extrinsic since they embed the target manifold $N$ into an ambient Euclidean space $\Real^K$ and compare two solutions $u_1, u_2:M\to N\hookrightarrow \Real^K$ by directly taking their difference.

A more intrinsic method was applied by McGahagan~\cite{Mc} to show that the uniqueness of the \sch flow actually holds in a larger function space. More precisely,  suppose $M=\Real^m$ is the Euclidean space, $N$ is a complete manifold with bounded geometry which is embedded into an Euclidean space $\Real^K$ and let $\S'_m$ be the function space
\[\begin{aligned}
\S'_m&=W^{[\frac{m+3}{2}],2}\cap \dot{W}^{1,\infty}\cap \dot{W}^{2,m}(\Real^m,N)\\
&=\{u:\Real^m\to N\hookrightarrow\Real^K|\norm{u}_{W^{[\frac{m+3}{2}],2}}+\norm{Du}_{L^\infty}+ \norm{D^2u}_{L^m}<\infty\},
\end{aligned}\]
where $D$ is the standard derivative on functions $u:\Real^m\to \Real^K$ and the homogeneous Sobolev space $\dot{W}^{k.p}$ consists of $k$-times weakly differentiable functions $u$ such that $D^\al u\in L^p$ for $|\al|=k$.. Then by comparing the derivative of two solutions via parallel transportation on $N$, McGahagan proved that a solution to the \sch flow in the space $L^\infty([0,T],\S'_m)$ is unique. Here a complete Riemannian manifold is said to have bounded geometry if it has positive injectivity radius and the Riemannian curvature tensor is bounded and has bounded derivatives. By Sobolev embedding theorems, it is easy to see that $W^{[m/2]+2,2}(\Real^m,N)\hookrightarrow \S'_m$. Thus it follows from the existence results that, if $u_0\in W^{k,2}(\Real^m,N)$ for $k\ge [m/2]+2$, then there exists a unique solution $u\in L^\infty([0,T),W^{k,2}(\Real^m,N))$ to the \sch flow (\ref{e:sch0}).

It is natural to ask if the uniqueness of the \sch flow holds for a general complete Riemannian manifold. In~\cite{RRS}, Rodnianski, Rubinstein and Staffilani asserts that for a complete Riemannian manifold $M$ and a complete K\"ahler manifold $N$ both with bounded geometry, the existence and uniqueness of a solution in $C^0([0, T], W^{k,2}(M,N))$ with initial data in $W^{k,2}(M,N)$ for $k\ge [m/2]+2$ follows directly from the work of Ding-Wang~\cite{DW2} and McGahagan~\cite{Mc}. However, a detailed proof is still missing in the literature. In this paper, by exploring the geometric ideas of McGahagan's proof, we obtain the following uniqueness results on complete manifolds.

To state our results, we define the function spaces
\[\begin{aligned}
\S_\infty &=W^{2,2}\cap \dot{W}^{1,\infty}\cap \dot{W}^{2,\infty}(M,N),\\
\S_m &=W^{[\frac{m}{2}]+1,2}\cap \dot{W}^{1,\infty}\cap \dot{W}^{2,m}(M,N).
\end{aligned}\]

\begin{thm}\label{t:main1}
Suppose $M$ is an $m$ dimensional complete manifold with bounded Ricci curvature $Ric_M$, $N$ is a complete K\"ahler manifold with bounded geometry. If $u_1, u_2\in L^\infty([0,T],\S_\infty)$ are two solution to the \sch flow~(\ref{e:sch0}) with the same initial map $u_0\in \S_\infty$, then $u_1=u_2$ a.e. on $[0,T]\times M$.
\end{thm}

\begin{thm}\label{t:main2}
Suppose $m\ge 3$, $M$ is an $m$ dimensional complete manifold with bounded Riemannian curvature $R_M$ and positive injectivity radius $inj(M)>0$, $N$ is a complete K\"ahler manifold with bounded geometry. If $u_1, u_2\in L^\infty([0,T],\S_m)$ are two solution to the \sch flow~(\ref{e:sch0}) with the same initial map $u_0\in \S_m$, then $u_1=u_2$ a.e. on $[0,T]\times M$.
\end{thm}

\begin{rem}
For a complete Riemannian manifold $N$ with bounded geometry, the above spaces $\S_\infty$ and $\S_m$ can by defined equivalently with or without referring to a embedding $N\hookrightarrow \Real^K$. Namely, we may define the Sobolev space of maps from $M$ to $N$ intrinsically by the covariant derivatives induced from the Levi-Civita connections on $M$ and $N$. Note that the index $[\frac{m+3}{2}]$ equals $[m/2]+1$ when $m$ is even and equals $[m/2]+2$ when $m$ is odd, thus the space $\S_m$ is larger than the space $\S'_m$ in McGahagan~\cite{Mc}.
\end{rem}

\begin{rem}
The assumptions on $M$ in Theorem~\ref{t:main2} is made to ensure the validity of Sobolev inequalities, in particular the Gagliardo-Nirenberg interpolation inequalities(cf. \cite{C,H}). Otherwise, we need the $L^\infty$ bound of the second derivatives of the solutions as in Theorem~\ref{t:main1}.
\end{rem}

The two theorems are proved simultaneously in Section~\ref{s:proof} and the proof is based on a geometric energy method. Given two solutions $u_1$ and $u_2$, we will define an energy functional which describes their difference up to the first-order derivatives as follows:
\[ Q(t):=\int_{\{t\}\times M}|d_N(u_1,u_2)|^2dv + \int_{\{t\}\times M}|\P \n u_2-\n u_1|^2dv, ~~t\in[0,T].\]
The functional consists of two parts. The first part is simply the integral of the distance of $u_1$ and $u_2$ on $N$. The second part is defined by the intrinsic distance of the differentials $\n u_1$ and $\n u_2$. The key step here is to construct a global isomorphism $\P$ between the two pull-back bundles $u_1^*TN$ and $u_2^*TN$ by using parallel transportation in $N$. The existence of $\P$ is guaranteed by the assumptions of the theorems. Then our goal is to show that this functional satisfies a Gronwall type inequality and hence vanishes identically on $[0,T]$.

The key innovation of McGahagan~\cite{Mc} is to compare $\n u_1$ and $\n u_2$ intrinsically via parallel transportation on $N$. While for the zeroth-order term, she still use the embedding $N \hookrightarrow \Real^K$ and the extrinsic distance $|u_1-u_2|_{\Real^K}$. Here we go one step further and use the intrinsic distance $d_N(u_1,u_2)$ instead. In this way the functional $Q$ is defined intrinsically. Actually, for Theorem~\ref{t:main1}, we provide a purely intrinsic proof.

One advantage of our method is that the derivatives of $d_N(u_1,u_2)$ is naturally connected with the first order term $\P \n u_2-\n u_1$. Correspondingly, the cost is that we need an estimate of the Hessian of the distance function, which appears in many other uniqueness problems in geometric analysis. It is interesting that, different from the uniqueness arguments of harmonic maps~\cite{JK,SY,CJW} and other parabolic geometric flows~\cite{CY,CZ},  we need an upper bound of the Hessian instead of a lower bound.

%Moreover, there is a small gap in McGahagan's proof concerning the regularity of the distance function (see Remark~\ref{r:gap}) and we bypass it by applying a recent result of Chen-Jost-Wang~\cite{CJW}, which states the equivalence of two kinds of pseudo distances of tangent vectors.

Another feature of our presentation is that we use the method of moving frames and gauge language to illustrate the geometric ideas more clearly. For example, we give an explicit expression of the difference of pull-back connections and corresponding Laplacian operators.

The energy method can also be used in proving uniqueness of other types of geometric flows. For example, Kotschwar applied the energy method to prove the uniqueness of Ricci flow~\cite{K1,K2}. Part of our motivation of the current work arises from our study of another \sch type geometric flow, namely, the Skew Mean Curvature Flow(SMCF)~\cite{SS}. The Gauss map of SMCF satisfies a coupled system consisting of the \sch flow and a metric flow, where the metric on the domain manifold of the \sch map evolves along time~\cite{S}. The uniqueness of SMCF is still open and our method here provides a possible solution to the more challenging problem.

\section*{Acknowledgement}

Part of this work was carried out when the first author was visiting University of British Columbia and University of Washington. He would like to thank Professor Jingyi Chen and Professor Yu Yuan for their generous help and support.

\section{Preliminaries}

\subsection{\sch Flow in Moving Frame}

Let $T>0$ and $I=[0,T]$ be an interval. Suppose $u:I\times M\to N$ is a solution to the \sch flow
\begin{equation}\label{e:sch}
  \p_tu=J(u)\tau(u).
\end{equation}
We are going to rewrite the above equation in a moving frame, namely, a chosen gauge of the pull-back bundle $u^*TN$.

To fix our notations, we let roman numbers $i,j,k$ be indices ranging from $1$ to $m$, bold ones $\ii, \jj, \kk$ ranging from $0$ to $m$, and Greek letters $\al,\be$ ranging from $1$ to $n$, where $n$ is the dimension of $N$. Let $\b{M}:=I\times M$ be endowed with the natural product metric. We will use $\n$ to denote connections on different vector bundles which are naturally induced by the Levi-Civita connections on $M$ and $N$. In particular, this includes the pull-back bundle $u^*TN$ on $\b{M}$, the pull-back bundle $u(t)^*TN$ on some time slice $\{t\}\times M$ for $t\in I$ and their tensor product bundles with the cotangent bundle $T^*M$. Sometimes in the context, we also use more specific notations such as $\n^N$ and $\n^M$ to emphasize which connection we are using.

Locally on an open geodesic ball $U\subset M$, we may choose an orthonormal frame $\{e_i\}_{i=1}^m$ of the tangent bundle $TM$. Set $e_0:=\p_t$ such that $\{e_\ii\}_{\ii=0}^m$ forms a local orthonormal basis of $T(I\times U)$. For convenience, we denote $\n_\ii:=\n_{e_\ii}$ and $\n_t=\n_0$. Then $\n_t e_\ii=0, 0\le \ii\le m$ with $\n$ the Levi-Civita connection on $\b{M}$.

Recall that the tension field is $\tau(u)=\tr_g \n^2u = \n_k\n_k u$, where $\n_k u$ denotes the covariant derivative of $u$ and is a section of the bundle $u^*TN\otimes T^*M$. Then the \sch flow (\ref{e:sch}) has the form
\begin{equation*}
  \n_tu=J(u)\n_k\n_ku.
\end{equation*}
Differentiating the equation, we get
\begin{equation*}
  \begin{aligned}
    \n_t\n_i u&=\n_i\n_t u\\
    &=\n_i(J(u)\n_k\n_k u)=J(u)\n_i\n_k\n_ku\\
    &=J(u)(\n_k\n_i\n_k u + R^N(\n_i u, \n_k u)\n_k u + R^M(e_i, e_k, e_k, e_l)\n_l u)\\
    &=J(u)(\n_k\n_k\n_i u + R^N(\n_i u, \n_k u)\n_k u + Ric^M(e_i, e_l)\n_l u),
  \end{aligned}
\end{equation*}
where $R^M, R^N$ are the curvature of $M$ and $N$, respectively, and $Ric^M$ is the Ricci curvature of $M$. Here we have used the fact that $\n$ is torsion free and $\n^N J=0$ since the target manifold $N$ is K\"ahler.

Next we choose a local frame $\{f_\al\}_{\al=1}^{n}$ of the pull-back bundle $u^*TN$, such that the complex structure $J$ in this frame is reduced to a constant skew-symmetric matrix which we denote by $J_0$. Letting $\n_\ii u=:\phi_\ii^\al f_\al$, we may further rewrite the above equation for $\n_i u$ as
\begin{equation}\label{e:sch1}
  \n_t\phi_i=J_0(\Delta_x\phi_i+R^N(\phi_i, \phi_k)\phi_k + Ric^M_{ij}\phi_j),
\end{equation}
where $\Delta_x=\n_k\n_k$ is the Laplacian operator on $u(t)^*TN\otimes T^*M$.

%Next we let $\om^\ii$ denote the orthonormal co-frame which is the dual of $e_\ii$. Then in this frame, the pull-back connection on $u^*TN$ can be expressed by $\n=d +A$, where $A=A_\ii\om^\ii$ is a matrix valued 1-form. The corresponding curvature from is given by
%\[F=dA+[A,A]=F_{\ii\jj}\om^\ii\wedge\om^\jj.\]
%On the other hand, $F$ is just the pull-back of $R$ on $N$. Equivalently, if $R=R_{\al\be}\th^\al\wedge \th^\be$ where $\th^\al$ is the dual of $f_\al$, then
%\begin{equation}\label{e:sch2}
%F=u^*R=R_{\al\be}\phi_\ii^\al\phi_\jj^\be\om^\ii\wedge\om^\jj.
%\end{equation}
%Combining equation~(\ref{e:sch1}) and (\ref{e:sch2}), we arrive at
%\begin{equation}\label{e:sch3}
%\left\{\begin{aligned}
%    \n_t\phi_i&=J_0\Delta_x\phi_i + JF_{ik}\phi_k,\\
%    F_{\ii\jj}&=R_{\al\be}\phi_\ii^\al\phi_\jj^\be.
%\end{aligned}\right.
%\end{equation}

\subsection{Pseudo distance of tangent vectors}\label{s:distance}

In the following context, we always assume $N$ is a complete manifold with curvature bounded by $K_0$ and injectivity radius bounded from below by $i_0>0$. Let $\de_0=\min\{\frac{i_0}{2},\frac{1}{4\sqrt{K_0}}\}$ and $D\subset N$ be an open ball with radius $\de_0$. Then for any $y_1, y_2\in D$, there exists a unique minimizing geodesic $\ga:[0,1]\to D$ connecting $y_1$ and $y_2$. Let $\P:T_{y_2}N \to T_{y_1}N$ be the linear map given by parallel transportation along $\ga$.

For two vectors $X_\ld\in T_{y_\ld}N, \ld=1,2$, there is a natural distance function defined by
\[\mathfrak{d}_0(X_1, X_2):=|\P X_2-X_1|.\]
On the other hand, we can find a Jacobi field $\b{X}$ along $\ga$ such that $\b{X}(0)=X_1$ and $\b{X}(1)=X_2$. There is another distance function given by (cf. \cite{JK})
\[\mathfrak{d}(X_1, X_2):=\left\{\begin{aligned}
  &\(\int_0^1|\n_s \b{X}|^2ds\)^{\frac12}, \quad y_1\neq y_2;\\
  &|X_1-X_2|, \quad y_1=y_2.
\end{aligned}.\right.\]

It turns out that the two distance functions are in a sense equivalent. Here we quote the following lemma of Chen, Jost and Wang~\cite{CJW}.

\begin{lem}\label{l:distance}
There exists a constant $C$ depending on the geometry of $N$ such that
\[ |\mathfrak{d}_0(X_1, X_2)-\mathfrak{d}(X_1, X_2)|\le C(|X_1|+|X_2|)d(y_1,y_2),\]
where $d$ is the distance function of $N$.
\end{lem}

\subsection{Hessian of distance function}

The distance function $d$ of $N$ can be regarded as a function defined on $N\times N$. It is well-known that its square $d^2$ is smooth when restricted to $D\times D$. Let $\tn:=\n\oplus\n$ be the covariant derivative on $N\times N$ induced by the Levi-Civita connection $\n$ on $N$.

\begin{lem}\label{l:Hessian}
Let $\t{X}=(X_1,X_2), \t{Y}=(Y_1, Y_2)$ be two vectors in $T_{y_1}N\times T_{y_2}N$, then
  \begin{align*}
    \frac{1}{2}\tn d^2(\t{X})&=\<\ga'(0),\P X_2-X_1\>,\\
    \frac{1}{2}|\tn^2d^2(\t{X},\t{Y})|&\le |\P X_2-X_1||\P Y_2-Y_1|+Cd^2(|X_1|+|X_2|)(|Y_1|+|Y_2|).
  \end{align*}
\end{lem}
\begin{proof}
Let $\b{T}:=\ga'/d$ denote the unit tangent vector along $\ga$ and $T_1=\b{T}(0), T_2=\b{T}(1)$. Since $\b{T}$ is parallel along $\ga$, we have $\P(T_2)=T_1$. By the formula of gradient of distance function, we have
\[\begin{aligned}
  \tn d(\t{X})&=\<-T_1, X_1\>+\<T_2, X_2\>\\
  &=\<T_1,-X_1\>+\<\P T_2, \P X_2\>\\
  &=\<T_1,\P X_2-X_1\>.
\end{aligned}\]
This proves the first identity of the lemma.

For the Hessian estimate, we have
\[\frac{1}{2}\tn^2d^2(\t{X},\t{Y})=\tn d(\t{X})\cdot\tn d(\t{Y})+d\tn^2d(\t{X},\t{Y}).\]
Let $\b{X}$ be the Jacobi field along $\ga$ with $\b{X}(0)=X_1$ and $\b{X}(1)=X_2$. Similarly, Let $\b{Y}$ be the Jacobi field along $\ga$ with $\b{Y}(0)=Y_1$ and $\b{Y}(1)=Y_2$. Recall the second variational formula of distance function (see for example Theorem 5.4 of \cite{KN})
\[
\tn^2d(\t{X},\t{Y})=\frac{1}{d}\(\int_0^1\<\n_{s}\b{X}^\bot, \n_{s}\b{Y}^\bot\>ds-\int_0^1 \<R(\ga',\b{X}^\bot)\b{Y}^\bot,\ga'\>ds\),
\]
where $\b{X}^\bot, \b{Y}^\bot$ is the component of $\b{X}, \b{Y}$ perpendicular to $\b{T}$. It follows that
\[\begin{aligned}
\frac{1}{2}\tn^2d^2(\t{X},\t{Y})&= \<T_1,\P X_2-X_1\>\<T_1, \P Y_2-Y_1\>\\
  &\quad +\int_0^1\<\n_{s}\b{X}^\bot, \n_{s}\b{Y}^\bot\>ds-d^2\int_0^1 \<R(\b{T},\b{X}^\bot)\b{Y}^\bot,\b{T}\>ds.
  \end{aligned}\]
By Lemma~\ref{l:distance}, the second term in the right hand side can be bounded by
\[\left|\int_0^1\<\n_{s}\b{X}^\bot, \n_{s}\b{Y}^\bot\>ds\right|\le |\P X_2^\bot-X_1^\bot||\P Y_2^\bot-Y_1^\bot|+Cd^2(|X_1|+|X_2|)(|Y_1|+|Y_2|).\]
Moreover, by the equation for the Jacobi field, it is easy to see that (see for example the proof of Lemma~\ref{l:Jacobi} below)
\[|\b{X}|\le C(|X_1|+|X_2|)\]
and
\[|\b{Y}|\le C(|Y_1|+|Y_2|).\]
Thus we have
\[\left|\int_0^1 \<R(\b{T},\b{X}^\bot)\b{Y}^\bot,\b{T}\>ds\right|\le C(|X_1|+|X_2|)(|Y_1|+|Y_2|).\]
Combining the above inequalities together, we get the second identity and hence finish the proof of the lemma.
\end{proof}

\section{Proof of Uniqueness}\label{s:proof}

\subsection{Outline of the proof}

In this section we prove Theorem~\ref{t:main1} and Theorem~\ref{t:main2} simultaneously. For $\S=\S_\infty$ or $\S=\S_m$, let $u_1, u_2\in L^\infty([0,T), \S)$ be two solutions to the \sch flow~(\ref{e:sch0}) with same initial value $u_0\in \S$. We need to show that $u_1=u_2$ a.e. for all $(t,x)\in [0,T)\times M$.

%A direct proof of uniqueness is given by Ding-Wang \cite{DW1} for $C^3$ solutions. The proof relies on the embedding of the target manifold $N$ into an Euclidean space and computes the difference of two solutions directly. McGahagan \cite{Mc} gives another proof using parallel transportation to compare two solutions. This method is more intrinsic and results in better estimates. In fact, she proves the uniqueness in the space $W^{k,2}\cap \dot{W}^{1,\infty}\cap \dot{W}^{2,n}$ where the domain manifold is $M=\Real^n$ and $k=[\frac{n+3}{2}]$. Here we follow her method closely.

The first step is to construct a family of geodesics connecting the two solutions and hence a globally defined parallel transportation. To this order, we show that in a sufficiently small time interval $I:=[0,T']$, the two solutions lie sufficiently close to each other, such that there exists an unique geodesic connecting $u_1$ and $u_2$ for each $(t,x)\in I\times M$. More precisely, we define a map $U:[0,1]\times I\times M\to N$ such that $U(0,t,x)=u_1(t,x)$, $U(1,t,x)=u_2(t,x)$ and $\ga_{(t,x)}(s):=U(s,t,x):[0,1]\to N$ is a geodesic for any fixed $(t,x)\in I\times M$. Thus we can define a linear map $\P:u_2^*TN\to u_1^*TN$ between the two pull-back bundles by using parallel transportations along the geodesics.

Next we define two functions
\[Q_1(t):=\int_{M}|d(u_1,u_2)|^2dv,\]
and
\[Q_2(t):=\int_{M}|\P \n u_2-\n u_1|^2dv.\]
Our goal is to derive a Gronwall type estimate for the energy $Q_1(t)+Q_2(t)$ and conclude that $Q_1(t)=Q_2(t)=0$ for all $t$. Two estimates will play an important role in the computation. The first is the estimate of the Hessian of distance function. The second one is the estimate of difference of pull-back connections corresponding to the two solutions.

\subsection{Construction of connecting geodesics}

First we need the following lemma.

\begin{lem}\label{l1}
Under the assumptions of Theorem~\ref{t:main1} or Theorem~\ref{t:main2}, there exists $T'>0$ such that $d(u_1,u_2)<\de_0$ for any $(t,x)\in [0,T']\times M$.
\end{lem}
\begin{proof}
To prove the lemma, we only need to show that for both $\ld=1$ and $2$, $u_\ld(t,x)$ stays close to $u_0(x)$ for fixed $x\in M$ and sufficiently small $t>0$.

If $u_\ld$ satisfies the assumptions of Theorem~\ref{t:main1}, then $\tau(u_\ld) \in L^\infty([0,T]\times M))$. In this case, we can simply bound the distance of $u_0(x)=u_\ld(0,x)$ and $u_\ld(t,x)$ by the length of the curve $\ga_\ld(\cdot)=u_\ld(\cdot, x)$. In fact, from the equation $\p_tu_\ld=J(u_\ld)\tau(u_\ld)$, we deduce
\[d(u_\ld(t,x), u_0(x)\le \int_0^t|\p_t u_\ld|dt\le t\norm{\tau(u)}_{L^\infty}\le Ct.\]
Thus the lemma holds for Theorem~\ref{t:main1}.

For the case of Theorem~\ref{t:main2}, we need to embed the target manifold $N$ into an Euclidean space. By the \sch flow equation, we have
\[ \frac{1}{2}\frac{d}{dt}\norm{u_\ld(t,x)-u_0(x)}^2_{L^2}
=\int_M\<u_\ld-u_0, \p_t u_\ld\> dv
\le C\norm{u_\ld-u_0}_{L^2}\norm{\tau(u_\ld)}_{L^2}
\le C.\]
Since $u_\ld(0,x)=u_0(x)$, it follows
\[\norm{u_\ld-u_0}_{L^2}\le Ct^{1/2}.\]
The assumptions on $M$ allows us to apply the Gagliardo-Nirenberg interpolation inequality (Theorem 5 in \cite{C}) to get
\[ \norm{u_\ld-u_0}_{L^\infty}\le C\norm{u_\ld-u_0}_{L^2}^a\norm{u_\ld-u_0}_{W^{[m/2]+1,2}}^{1-a}\le Ct^{a/2},\]
where $0<a=1-\frac{m}{2([m/2]+1)}<1$.

Note that the above bound only gives an estimate of the extrinsic distance of $u_\ld$ and $u_0$. However, since $u_\ld \in L^\infty([0,T],W^{[m/2]+1,2})$ and $\p_t u_\ld \in L^\infty([0,T],W^{[m/2]-1,2})$, by Sobolev embedding and interpolation inequalities, we know that $u_\ld$ actually belongs to $C^0([0,T]\times M,N)$. Thus for sufficiently small $T'>0$ and fixed $x\in M$, the curve $u_\ld(\cdot,x)$ lies in a connected neighborhood of $u_0(x)$ in $M$ which locates inside a small ball of the extrinsic Euclidean space. Since $N$ has bounded geometry, it follows
\[d(u_\ld, u_0)\le C\norm{u_\ld-u_0}_{C^0}=C\norm{u_\ld-u_0}_{L^\infty}\le Ct^{a/2},\]
Consequently, the lemma also holds for Theorem~\ref{t:main2}.
\end{proof}

An important fact is that the uniqueness is a local property. Namely, once we know $u_1=u_2$ on a small time interval $[0,T']$, then we can prove $u_1=u_2$ on the whole interval $[0,T]$ by repeating the argument. Therefore, we only need to prove Theorem~\ref{t:main1} and \ref{t:main2} in the time interval $I=[0,T']$.

Now by Lemma~\ref{l1}, for any $(t,x)\in I\times M$, there exists a unique minimizing geodesic $\ga_{(t,x)}:[0,1]\to N$ such that $\ga_{(t,x)}(0)=u_1(t,x)$ and $\ga_{(t,x)}(1)=u_2(t,x)$. By letting $(t,x)$ vary, the family of geodesics give rise to a map $U:[0,1]\times I\times M\to N$ connecting $u_1$ and $u_2$, where $U(s,t,x)=\ga_{(t,x)}(s)$. Therefore, we can define a global bundle morphism $\P:u_2^*TN\to u_1^*TN$ by the parallel transportation along each geodesic. Moreover, $\P$ can be extended naturally to a bundle morphism from $u_2^*TN\otimes T^*M$ to $u_1^*TN\otimes T^*M$.

\subsection{Estimate of $Q_1$}

Now consider the composition of the distance function $d:N\times N\to \Real$ and $\t{u}:=(u_1,u_2):I\times M\to N\times N$. Let $\t{X}=(\n u_1, \n u_2)$ and $\t{Y}=(J\n u_1, J\n u_2)$. Using the \sch flow equation and integrating by parts, we have
\[\begin{aligned}
  \frac{d}{dt}Q_1&=\frac{d}{dt}\int_M|d(u_1,u_2)|^2dv\\
  &=\int_M \<\tn d^2, (\p_t u_1,\p_t u_2)\>dv\\
  &=\int_M \<\tn d^2, (J\tau(u_1), J \tau(u_2))\>dv\\
  &=-\int_M \<\n \tn d^2, (J \n u_1, J \n  u_2)\>dv\\
  &=-\int_M \tn^2 d^2(\t{X}, \t{Y}) dv.
\end{aligned}\]
Applying Lemma~\ref{l:Hessian}, we have
\[\begin{aligned}
  \frac{1}{2} |\tn^2 d^2(\t{X}, \t{Y})|
  &\le  |\P X_2-X_1||\P Y_2-Y_1|+Cd^2(|X_1|+|X_2|)(|Y_1|+|Y_2|)\\
  &= |\P \n u_2-\n u_1||\P J\n u_2-J\n u_1|+Cd^2(|\n u_1|+|\n u_2|)(|J\n u_1|+|J\n u_2|)\\
  &\le  |\P \n u_2-\n u_1|^2+Cd^2.
\end{aligned}\]
Therefore, we arrive at
\begin{equation}\label{e:Q1}
  \frac12\frac{d}{dt}Q_1\le Q_2+CQ_1,
\end{equation}
where the constant $C$ depends on $L^\infty$ norm of $\n u_1$ and $\n u_2$.

\subsection{Estimate of $Q_2$}

Next we derive estimates for the functional
\[Q_2=\int_M  |\P \n u_2-\n u_1|^2 dv.\]
To proceed, we express the bundle morphism $\P$ more explicitly by choosing local orthonormal frames of the pull back bundles. In particular, we can arrange the frame to be parallel along the connecting geodesics which is constructed in the previous section.

More precisely, we first fix a local orthonormal frame on $u_1^*TN$. For each point $(t,x)$, we parallel transport the frame to get a moving frame $\{\b{f}_{\al}(s)\}$ along the geodesic $\ga_{(t,x)}(s)$. Then we set $f_{1,\al}=\b{f}_{\al}(0), f_{2,\al}=\b{f}_{\al}(1)$. Obviously, by the construction, we have $\P f_{2,\al}=f_{1,\al}$. If we denote $\n_i u_\ld =\phi_{\ld,i}^\al f_{\ld,\al}$, it follows
\[ \P\n_i u_2= \P(\phi_{2,i}^\al f_{2,\al})=\phi_{2,i}^\al f_{1,\al}.\]
Thus, letting $\phi_\ld:=\phi_{\ld, i}^\al$ and $\psi:=\phi_2-\phi_1$, the quantity we need to consider is simply
\[Q_2=\int_M  |\phi_{2}- \phi_{1}|^2 dv=\int_M  |\psi|^2 dv.\]
In other words, by using the bundle morphism $\P$, we actually regard $\phi_\ld, \ld=1,2$ as sections living on the same bundle $u_1^*TN$. However, the pull-back connection $\n_\ld:=u_\ld^*\n^N$ acting on $\phi_\ld$ stays distinct.

Recall that by (\ref{e:sch1}), $\phi_\ld$ satisfies the following equation
\begin{equation}\label{e:sch4}
  \n_{\ld,t}\phi_{\ld}=J_0\Delta_{\ld}\phi_\ld + J_0R^N\#\phi_\ld\#\phi_\ld\#\phi_\ld+ J_0Ric^M\#\phi_\ld,
\end{equation}
where $\#$ denotes linear combinations of the components of involved terms. Since $\phi_1$ and $\phi_2$ are now regarded as sections on the same bundle, we may subtract (\ref{e:sch4}) for $\ld=1, 2$ to get
\[
  \n_{1,t}\psi+ (\n_{2,t}-\n_{1,t})\phi_2
  =J_0\Delta_{1,x}\psi+J_0(\Delta_{2,x}-\Delta_{1,x})\phi_2 + S,
\]
where
\[ S:= J_0(R^N(u_1)\#\phi_1\#\phi_1\#\phi_1-R^N(u_2)\#\phi_2\#\phi_2\#\phi_2)+ J_0Ric^M\#(\phi_1 - \phi_2).\]
Hence we have
\[\begin{aligned}
\frac12\frac{d}{dt}Q_2&=\int_M\<\psi, \n_{1,t}\psi\>dv\\
&=\int_M\<\psi,J_0\Delta_{1,x}\psi\>dv+\int_M\<\psi, S\>dv\\
&\quad +\int_M\<\psi, -(\n_{2,t}-\n_{1,t})\phi_2+J_0(\Delta_{2,x}-\Delta_{1,x})\phi_2\>dv.
\end{aligned}\]
The first term vanishes after integration by parts. Moreover, by the assumption of bounded geometry of $N$, it is easy to see that
\[ |S| \le C(d(u_1, u_2)+|\psi|),\]
where the constant depends on $R^N, \n^NR^N, Ric^M$ and the $L^\infty$-norm of $\phi_\ld$.
Thus we arrive at
\begin{equation}\label{e:Q2}
  \begin{aligned}
\frac12\frac{d}{dt}Q_2&\le
\int_M|\psi|\Big(|(\n_{2,t}-\n_{1,t})\phi_2|+|J_0(\Delta_{2,x}-\Delta_{1,x})\phi_2| +C|d|+C|\psi|\Big)dv\\
&\le C( \norm{d}_{L^2}^2+\norm{\psi}_{L^2}^2+\norm{(\n_{2,t}-\n_{1,t})\phi_2}_{L^2}^2+\norm{(\Delta_{2,x}-\Delta_{1,x})\phi_2}_{L^2}^2).\\
\end{aligned}
\end{equation}
Therefore, we are led to compute the difference of the two connections and corresponding Laplacians.

\subsection{Estimate of the connection}

Denote the difference of the two connections $\n_\ld = u_\ld^*\n^N$, which is a tensor, by
\[B:=\n_2-\n_1.\]
To be more specific, let $\om^\ii$ be the orthonormal co-frame which is the dual of $e_\ii$. Under the local frame $\{f_{\ld,\al}\}$ of the pull-back bundle $u_\ld^*TN$, we can write $\n_\ld=d+A_\ld$ where $A_\ld=A_{\ld,\ii}\om^\ii$ is a (skew-symmetric) matrix valued 1-form. Thus $B:=B_\ii\om^\ii=(A_{2,\ii}-A_{1,\ii})\om^\ii$.

Recall that by our construction, we have a map $U:[0,1]\times I\times M\to N$ such that $U(s,t,x):=\ga_{(t,x)}(s)$ is a geodesic. Thus we have a global pull-back bundle $U^*TN$, which is defined over $[0,1]\times I\times M$, with an orthonormal frame $\{\b{f}_\al\}$ which is defined by parallel transportation. Now let $\bn:=U^*\n^N$ denote the pull-back connection on $U^*TN$ which corresponds to an 1-form
$$\bar{A}=\bar{A}_sds+\bar{A}_\ii\om^\ii.$$
In particular, since $\b{f}_\al$ is parallel along $s$, $\bar{A}_s$ vanishes, leaving along $\bar{A}=\bar{A}_\ii\om^\ii$. The curvature of $\bn$ is given by the formula
\[\b{F}=d\b{A}+[\b{A},\b{A}].\]
Since $\b{A}_s=0$, the $ds\wedge \om^\ii$ component of $\b{F}$ is simply
\[ \b{F}_{s\ii}=\p_s\b{A}_\ii.\]

On the other hand, we have $U(0,t,x)=u_1(t,x)$, $U(1,t,x)=u_2(t,x)$. Obviously, $u_1^*TN$ and $u_2^*TN$ are just the restriction of $U^*TN$ at $s=0$ and $s=1$, respectively. Moreover, the restriction of $\bn$ at $u_\ld^*TN$ is just the pull-back connection $\n_\ld$. That is, $\bar{A}_\ii(0)=A_{1,\ii}$ and $\bar{A}_\ii(1)=A_{2,\ii}$. Therefore,
\[ B_\ii=A_{2,\ii}-A_{1,\ii}=\int_0^1\b{F}_{s\ii}ds.\]
Note that $\b{F}$ is in fact the pull-back of the curvature $R^N$ on $N$, i.e. $\b{F}=U^*R^N$. It follows
\begin{equation}\label{e:connection1}
 B_\ii=\int_0^1R^N(\bn_s U, \bn_\ii U)ds.
\end{equation}
Since $|\bn_s U|=|\p_s \ga_{(t,x)}|=d(u_1,u_2)$, we have
\begin{equation}\label{e:connection2}
|B_{\ii}|\le \sup_s|R^N||\bn_s U||\bn_\ii U|\le C\sup_s|\bn_\ii U|d.
\end{equation}

Next we need the following lemma to estimate the Jacobi field $\bn_\ii U$ and its derivatives. See the appendix of \cite{Mc} for another proof.
\begin{lem}\label{l:Jacobi}
  The derivatives of $U$ satisfy the following estimates:
  \begin{equation*}
    \begin{aligned}
      &\sup_s|\bn_\ii U|\le C(|\phi_{1,\ii}|+|\phi_{2,\ii}|),\\
      &\sup_s|\bn_i\bn_j U|\le C(|\n_{1,i}\phi_{1,j}|+|\n_{2,i}\phi_{2,j}|)+C(|\phi_{1,i}|+|\phi_{2,i}|)(|\phi_{1,j}|+|\phi_{2,j}|),\\
    \end{aligned}
  \end{equation*}
  where the constant $C$ depends on the derivative (up to second order) of the exponential map in the domain.
\end{lem}
\begin{proof}
  First recall that each Jacobi field $W$ can be generated by a family of variation of geodesics
  $$\ga(s,t)=\exp_p(s(T+tV))$$
  and has the form
  \[W(s)=\p_t\ga(s,0)=D\exp_p|_{sT}(sV).\]
  Therefore, in a small geodesic ball we have
  \[|W(s)|\le|D\exp_p|_{sT}|\cdot s|V|\le s|D\exp_{sT}|_x|/|D\exp_p|_{T}|W(1)|.\]
  It follows that $|W(s)|\le C|W(1)|$ where the constant $C=\sup|D\exp_p|/\inf|D\exp_p|$. The constant can be achieved since we have $D\exp_p|_0=id$ and the exponential map is smooth.

  Now for any $0\le \ii\le n$, $W:=\p_\ii U$ is a Jacobi field along the geodesic connection $u_1$ and $u_2$, which satisfies
  \[\left\{\begin{aligned}
    &\n_s^2W+R(W,T)T=0,\\
    &W(0)=\phi_{1,\ii}, \quad W(l)=\phi_{2,\ii}.
  \end{aligned}\right.\]
  Since the Jacobi equation is linear, we can decompose $W=W_1+W_2$ where $W_1, W_2$ are both Jacobi fields such that $W_1(0)=0, W_1(1)=\phi_{2,\ii}$ and $W_1(0)=\phi_{1,\ii}, W_1(1)=0$, and both satisfy the above estimate. Therefore first desired inequality follows. The second one can be proved similarly by taking one more derivative.
\end{proof}

\subsection{Estimate of Laplacian}

Now we are ready to estimate the difference of two Laplacian operators $\Delta_{1,x}$ and $\Delta_{2,x}$. Since $\n_1=\n_2-B$, we have
\[\begin{aligned}
  \Delta_{1,x}
  &= \n_{1,k}\n_{1,k}= (\n_{2,k}-B_k)\circ(\n_{2,k}-B_k)\\
  &=\n_{2,k}\n_{2,k}-\n_{2,k}\circ B_k - B_k\n_{2,k} +B_k^2)\\
  &=\Delta_{2,x} - \n_{2,k} B_k - 2B_k\n_{2,k} + B_k^2.
  \end{aligned} \]
Hence
\[\Delta_{2,x}-\Delta_{1,x}=\n_{2,k} B_k + 2B_k\n_{2,k} - B_k^2.\]

The last two terms on the right hand side of the equality can be easily handled by (\ref{e:connection2}). To estimate the first term, first observe
\[  \n_{2,k} B_k= \bn_k B_k + (\n_{2,k}-\bn_k)B_k,\]
where we can control the last term by
\[\n_{2,k}-\bn_k=\b{A}_k(1)-\b{A}_k(s)=\int_s^1\b{F}_{sk}ds\le Cd.\]
So all we need to deal with is the term $\bn_k B_k$, which we use (\ref{e:connection1}) to estimate
\[\begin{aligned}
\bn_k B_k&=\bn_k\int_0^1R^N(\bn_s U, \bn_k U)ds\\
  &= \int_0^1\n^N R^N(\bn_k U,\bn_s U, \bn_k U)+R^N(\bn_s U, \bn_k^2 U) + R^N(\bn_k\bn_s U, \bn_k U)ds\\
  &\le C(\sup_s|\bn_k U|^2 d+\sup_s|\bn_k^2 U|d+\sup_s|\bn_k U|\int_0^1|\bn_k\bn_s U|ds).
  \end{aligned} \]
The first term can be controlled by applying Lemma~\ref{l:Jacobi}. As for the last term, we apply Lemma~\ref{l:distance} to derive
\[\begin{aligned}
\int_0^1|\bn_k\bn_s U|ds&=\int_0^1|\bn_s\bn_k U|ds\le \(\int_0^1|\bn_s\bn_k U|^2ds\)^{\frac 12}\\
&\le |\P\bn_kU(1)-\bn_kU(0)|+C(|\bn_kU(1)|+|\bn_kU(0)|)d\\
&\le |\phi_{2,k}-\phi_{1,k}|+C(|\phi_{2,k}|+|\phi_{1,k}|)d\\
&\le |\psi|+Cd.
  \end{aligned} \]
Consequently, we have
\begin{equation}\label{e:connection3}
|(\Delta_{2,x}-\Delta_{1,x})\phi_2|\le C(d+\sup_s|\bn_k^2 U|d+|\psi|),
\end{equation}
where the constant only depends on $|\phi_{\ld}|$ and the target manifold $N$.

\begin{rem}\label{r:gap}
Here we fixed a gap in McGahagan's proof by applying Lemma~\ref{l:distance}. In fact, the term $D\p_s\ga$, which appeared in the estimate of the derivative of the curvature term (line 18, page 394 in \cite{Mc}), may blow up if the two solutions are too close to each other. In particular, the distance function $d(u_1,u_2)$ is not differentiable at $x\in M$ if $u_1(t,x)=u_2(t,x)$.
\end{rem}

\subsection{Uniqueness}

Finally, we continue the estimate of $Q_2$ and finish the proof of Theorem~\ref{t:main1} and Theorem~\ref{t:main2}. Combining the estimates (\ref{e:Q2}), (\ref{e:connection2}) and (\ref{e:connection3}), we obtain
\begin{equation*}
  \frac12\frac{d}{dt}Q_2\le C(\norm{\psi}_{L^2}^2 + \norm{d}_{L^2}^2 + \norm{\sup_s(|\bn_k^2 U|+|\bn_t U|)d}_{L^2}^2).
\end{equation*}
Applying Lemma~\ref{l:Jacobi} again, we can bound the last term by
\[\norm{\sup_s(|\bn_k^2 U|+|\bn_t U|)d}_{L^2}^2\le C(\norm{(|\n_{1,k}\phi_{1,k}|+|\n_{2,k}\phi_{2,k}|)d}_{L^2}^2 + \norm{d}_{L^2}^2).\]

Now for Theorem~\ref{t:main1}, we have $\n^2u_\ld\in L^\infty(I\times M)$, then
\[\norm{(|\n_{1,k}\phi_{1,k}|+|\n_{2,k}\phi_{2,k}|)d}_{L^2}\le (\norm{\n_{1,k}^2u_1}_{L^\infty}+\norm{\n_{2,k}^2u_2}_{L^\infty})\norm{d}_{L^2}.\]
For Theorem~\ref{t:main2} where $\n^2u_\ld\in L^\infty(I,L^m(M,N))$, we have
\[\norm{(|\n_{1,k}\phi_{1,k}|+|\n_{2,k}\phi_{2,k}|)d}_{L^2}\le (\norm{\n^2u_1}_{L^m}+\norm{\n^2u_2}_{L^m})\norm{d}_{L^{\frac{2m}{m-2}}}.\]
By Sobolev embedding and the estimate in Lemma~\ref{l:Hessian}, we have
\[\norm{d}_{L^{\frac{2m}{m-2}}}\le C\norm{d}_{W^{1,2}}\le C(\norm{d}_{L^2} + \norm{\psi}_{L^2}).\]
In either case, we obtain
\begin{equation}\label{e:Q2f}
  \frac12\frac{d}{dt}Q_2\le C(Q_1+Q_2).
\end{equation}

The inequalities (\ref{e:Q1}) and (\ref{e:Q2f}) together yield
\begin{equation*}
  \frac12\frac{d}{dt}(Q_1+Q_2)\le C(Q_1+Q_2),
\end{equation*}
where $C$ depends on the norms of $u_1$ and $u_2$ in the space $\S$. Since $Q_1(0)=Q_2(0)=0$ at initial time, we conclude the uniqueness by Gronwall's inequality and finish the proof of Theorem~\ref{t:main1} and \ref{t:main2}.

%======================================================================

\end{document}